\documentclass[10pt, oneside, a4paper]{extarticle}
\usepackage[utf8]{inputenc}

\usepackage{amssymb}
\usepackage{amsmath}
\usepackage{amsthm}
\usepackage{url}
\usepackage{mathrsfs}
\usepackage{pdfpages}

\newtheorem{lem}{Lemma}

\newtheorem{theor}{Theorem}

\begin{document}

\author{Andrei Alpeev  \footnote{Chebyshev Laboratory, St. Petersburg State University, 14th Line, 29b, Saint Petersburg, 199178 Russia}}
\title{On Pinsker factors for Rokhlin entropy\footnote{This research is supported by the Chebyshev Laboratory  (Department of Mathematics and Mechanics, St. Petersburg State University)  under RF Government grant 11.G34.31.0026 and by JSC "Gazprom Neft"}}

\maketitle
\begin{abstract}
In this paper we will prove that any dynamical system posess the unique  maximal factor of zero Rokhlin entropy, so-called Pinsker factor. It is proven also, that if the system is ergodic and this factor has no atoms then system is relatively weakly mixing extension of its Pinsker factor. \\

\end{abstract}

keywords: Pinsker factor, Rokhlin entropy, generating partition, relatively weakly mixing extension  

\section{Introduction}


A few yesrs ago in work \cite{B10} Lewis Bowen introduced a new invariant for actions of sofic groups: a {\em sofic entropy}, and it lead to the great progress in the problem of classification of Bernoulli shifts up to measure conjugacy and even(for some class of groups) up to orbit equivalence. Sofic groups form a very large class of countable groups, but it is still an open question, whether all the countable groups are sofic or not.
 
So, sofic entropy, potentially, don't work for some groups. There is, altough a very natural alternative: minimal entropy of the generating partition, so-called {\em Rokhlin entropy}. The name comes from the result of Rokhlin, stating that Kolmogorov entropy for aperiodic transformations equals the infimum of the entropies of generating partitions \cite{Ro67}.
The fact that this is an invariant is obvious from the definition. It follows from \cite{B10} that Rokhlin entropy equals to the entropy of the base space for Bernoulli shifts over sofic groups. It is well known that sofic entropy is bounded from above by Rokhlin entropy(see \cite{B10},\cite{Ke13}).
In the work \cite{SeI} Seward proved a generalisation for Roklin entropy of Krieger and Denker theorems. Connection between the Rokhlin entropy, Gottschalk's surjunctivity conjecture and Kaplansky's direct finitness conjecture is established in \cite{SeII}.

In the realm of the Komogorov entropy it is well known that every system posess the unique maximal factor of zero entropy, so-called {\em Pinsker factor}, and it is known that the system isteslf is a relatively weakly mixing extension of its Pinsker factor(see \cite{Gl03}). In this paper we will transfer this classical results into the new setting of Rokhlin entropy. Namely, we will prove the theorem:
\begin{theor}\label{thm:pinsker_exists}
Every dynamical system contains a Pinsker subalgebra with respect to Rokhlin entropy, that is the unique maximal subalgebra  among subalgebras with zero Rokhlin entropy. 
\end{theor}

After it we will utilize Furstenberg-Zimmer structure theory in order to obtain the following result which resembles classical theorem about Kolmogorov entropy(\cite{Gl03}): 

\begin{theor}\label{thm:mixing}
Suppose that Pinsker factor of an ergodic dynamical system is nonatomic. Then the initial system is a realively weakly mixing extension of its Pinsker factor.
\end{theor}

{\bf Acknowledgements.} I would like to thank Mikl\'{o}s Ab\'{e}rt for his evening talk in the Arbeitsgemeinschaft on sofic entropy in Oberwolfach, october, 2013, which drew my attention to the Rokhlin entropy. I also would like to thank all the organizers and participants of the Arbeitsgemeinschaft. The term "Rokhlin entropy" is due to Mikl\'{o}s Ab\'{e}rt and Benjamin Weiss. I also would like to thank my advisor Anatoly Vershik for discusssions.

\section{Facts from ergodic theory}

Our intention here is to remind some facts and notions from ergodic theory we will use in proofs. 

A {\em standard probability space} or {\em Lebesgue space} is a measurable space arising from any Borel probability measure on standard Borel space. Such spaces have a lot of convenient properties(see \cite{Ro49}, \cite{Gl03}). 
Consider a standard probability space $\mathbf{X} = (X, \mathscr{X}, \mu)$ (there $X$ is a set, $\mathscr{X}$ is a $\sigma$-algebra and $\mu$ is a measure). A $\sigma$-subalgebra $\mathscr{A} \subset \mathscr{X}$ is said to be {\em complete} if for every $A \in \mathscr{A}$ and $B \in \mathscr{X}$ such that $\mu(A \Delta B)=0$ (there $\Delta$ denotes the symmetric difference operaion) we have $B \in \mathscr{A}$. Obviously, any $\sigma$-subalgebra has a {\em completion}.
Suppose $\pi: \mathbf{X} \to \mathbf{Y}$ is a measure-preserving map between standard probability spaces $\mathbf{X}=(X,\mathscr{X}, \mu)$ and $\mathbf{Y}=(Y,\mathscr{Y}, \nu)$. In this situation pair $\mathbf{Y}$ and $\pi$ is called a {\em factor}, $\pi$ is called {\em a factor-map}  and $\mathbf{Y}$ is called a factor-space.
There is a corespondent $\sigma$-subalgebra for any factor: consider the completion $\mathscr{Y}'$ of the $\sigma$-subalgebra $\lbrace \pi^{-1}(A) \vert A \in \mathscr{Y}\rbrace$. And also, for every complete $\sigma$-subalgebra in $\mathscr{X}$ there is such a factor that corespondent $\sigma$-subalgebra is exactly as given. So, there is a natural corespondence between $\sigma$-subalgebras and factors. for more details see \cite{Gl03}.

{\em Dynamical system} or {\em $G$-space} is a pair $(\mathbf{X},T)$ of a standard probability space $\mathbf{X}=(X, \mathscr{X}, \mu)$   and of an action of a countable group $G$ by measure preserving transformations. 
Two dynamical systems $(\mathbf{X}, T)$ and  $(\mathbf{Y}, S)$ with the same acting group are said to be {\em isomorphic} or {\em conjugated} if there is a meaure-preserving isomorphism $\varphi: \mathbf{X} \to \mathbf{Y}$ which is equivariant with respect to actions: $\varphi(T^g(x))= S^g(\varphi(x))$ for every $g$ and for almost every $x$.  A {\em factor} of the system $(\mathbf{X},T)$ is a pair of a system $(\mathbf{Y},S)$ and an equivariant measure-preserving map $\pi: \mathbf{X} \to \mathbf{Y}$. An {\em extension} of the system $(\mathbf{Y},S)$ is a pair of a system $(\mathbf{X},T)$ and an equivariant measure-preserving map $\pi: \mathbf{X} \to \mathbf{Y}$. Two extensioms $((\mathbf{X}_1, T_1), \pi_1)$ and $((\mathbf{X}_2, T_2), \pi_2)$ are said to be isomorphic if there is such an equivariant measure preserving a.e. bijection $\psi: \mathbf{X}_1 \to \mathbf{X}_2$ that $\pi_1 = \pi_2 \circ \psi$.

A $\sigma$-sibalgebra $\mathscr{A}$ is called invariant subalgebra if for every set $A \in \mathscr{A}$ and for every $g \in G$ we have $T^g(A) \in \mathscr{A}$ too. It is not hard to see that on the corespondent factor-space a factor-action $S$ of the group $G$ could be defined. So there is a natural corespondence between factors and invariant $\sigma$-subalgebras(see \cite{Gl03}).

Let $\mathbf{X}=(X, \mathscr{X}, \mu)$ be a standart probability space. A {\em partition} $\alpha$ is a countable or finite collection of disjoint measurable subsets of $\mathbf{X}$ covering all the space. A partition $\alpha$ is said to be measurable with respect to $\sigma$-subalgebra $\mathscr{A}$ if all its elements belong to $\mathscr{A}$.  
Consider a dynamical system $(\mathbf{X},T)$. We will say that $\sigma$-subagebra is generated by the partition if it is the smallest complete invariant $\sigma$-subalgebra with respect to that the partition is measurable. Partition is said to be {\em generating} if it generates the $\sigma$-algebra $\mathscr{X}$.

Let $(\mathbf{X}, T)$,$(\mathbf{Y},S)$ be two dynamical systems and let $\pi: \mathbf{X} \to \mathbf{Y}$ be a factor-map. For  any partition $\alpha$ of $\mathbf{X}$, measurable with respect to he $\sigma$-subalgebra, corespondent to the factor $(\mathbf{Y}, \pi)$, there is a unique  up to the set of measure $0$ partition $\beta$ of $\mathbf{Y}$ such that $\alpha=\pi^{-1}(\alpha)$. 

\begin{lem}\label{lem:generating}
Let $(\mathbf{X}, T)$ (where $\mathbf{X}=(X, \mathscr{X}, \mu)$) be a dynamical system, let $\mathscr{A}$ be an invariant  $\sigma$-subalgebra, and let $\alpha$ be a $\mathscr{A}$-measurable partition of $\mathbf{X}$. Let $\beta$ be such a partition of $\mathbf{Y}$ that $\alpha=\pi^{-1}(\beta)$.
Let $(\mathbf{Y}, S)$ be a corespondent factor-action and let $\pi: \mathbf{X} \to \mathbf{Y}$ be a factor-map.
The following are eqwuivalent: 
\begin{enumerate}
\item{$\alpha$ generates the $\sigma$-subalgebra $\mathscr{A}$}
\item{$\beta$ generates the $\sigma$-algebra $\mathscr{Y}$}
\item{$\beta$ separates points of $\mathbf{Y}$, that is there exists a conull set $Y' \subset Y$ such that for any $y_1, y_2 \in Y'$,  $y_1 \neq y_2$ there is such $g \in G$ that $S^g(y_1)$ and $S^g(y_2)$ lie in different parts of $\beta$. }
\end{enumerate}
\end{lem}

\section{Basic facts on Rokhlin entropy}\label{sect:basic}

All the $\sigma$-subalgebras in the sequel are complete.

Let $\mathbf{X}=(X,\mathscr{X}, \mu)$ be a standard probability space and let $(\mathbf{X},T)$ be a dynamical system.  
The Shannon entropy of the parition $\alpha = \lbrace A_1, A_2, \ldots \rbrace$ is given by the formula $H(\alpha)= - \sum_{i=0} \mu(A_i) \log(\mu(A_i))$ with the usual convention $0\log0=0$. The Rokhlin entropy of an invatiant $\sigma$-subalgebra $\mathscr{A}$ is the infimum of the Shannon entropies of generating partitions which are measurable with respect to this subalgebra. We will denote it by $h(\mathscr{A})$. The Rokhlin entropy $h(\mathbf{X},T)$ of a dynamical system is defined as Roklin entropy of $\sigma$-algebra $\mathscr{X}$.
We will say that $\sigma$-algebra $\mathscr{A}$ is generated by the set $\lbrace \mathscr{A}_i \rbrace$ of $\sigma$-algebras  if $\mathscr{A}$ is a minimal $\sigma$-algebra containing this subalgebras and we will denote it by $\mathscr{A}=\bigvee \mathscr{A}_i$. We will also say in this situation that $\mathscr{A}$ is join of algebras $\lbrace \mathscr{A}_i \rbrace$. 
\begin{lem}\label{subadditive inequality}
Suppose $\lbrace \mathscr{A}_1, \mathscr{A}_2, \ldots\rbrace$ is a countable or finite set of invariant $\sigma$-subalgebras and $\mathscr{A}$ is their join. Then $h(\mathscr{A}) \leq \sum{h(\mathscr{A}_i)}$.
\end{lem}
\begin{proof}   
Statement is obviously true if the sum of entropies is infinite.
Otherwise let us take any $\varepsilon>0$. By the definition, for any $i$ there is a generating for $\mathscr{A}_i$ partition $\alpha_i$ of Shannon entropy smaller than $h(\mathscr{A}_i)+\varepsilon/2^i$ . By completness of the space of partitions with respect to Rokhlin metric(see \cite{Ro67}), we have that there exists $\bigvee_i \alpha_i$ and $H(\bigvee_i \alpha_i) \leq \sum_i H(\alpha_i) < \sum_i h(\mathscr{A}_i) + \varepsilon$. It is obviously a generating partition for $\mathscr{A}$. 
\end{proof}

\begin{lem}\label{countable subsequence}
Suppose $I$ is a linearly ordered set and $\lbrace \mathscr{A}_i \rbrace_{i \in I}$ is a monotone sequence of $\sigma$-subalgebras, that is $\mathscr{A}_i \subseteq \mathscr{A}_j$ for $i \leq j$. Then there is a countable subset $J$ in $I$ such that $\bigvee_{i \in I}{\mathscr{A}_i}=\bigvee_{j \in J}{\mathscr{A}_i}$.
\end{lem}
\begin{proof}
The set $\mathscr{X}$ endowed with the metric $d(A,B)=\mu(A \Delta B)$ is a complete separable semimetric space. It is not hard to see that the join of the monotone sequence of $\sigma$-subalgebras corespond to the closure of the union of corespondent subset in the semimetric space $(\mathscr{X},d)$. Now we can refine a countable subsequence with the same closure of the union since our space is separable semimetric.
\end{proof}

\begin{proof}[Proof of theorem \ref{thm:pinsker_exists}]
Let us consider the set of all the invariant $\sigma$-subalgebras of zero Rokhlin entropy, ordered by the inclusion relation. By the two previous lemmas and Zorn's lemma we have that there is at least one maximal zero entropy algebra. It is easy to see that it is unique, since two maximal subalgebras could be joined, and their join will differ from both of them and have zero entropy, which lead to contradiction .
\end{proof}

\section{Pinsker factor and relative weak mixing}\label{sect:weak_mixing}

Suppose $(\mathbf{X}, T)$ ($\mathbf{X}=(X,\mathscr{X}, \mu)$ )is a $G$-space and $(\mathbf{Y}, R)$ ($\mathbf{Y}=(Y,\mathscr{Y}, \nu)$) is its factor. System $(\mathbf{X},T)$ is called relatively weakly mixing over  $(\mathbf{Y},S)$ if its relatively independent joining with any ergodic system over the common factor $(\mathbf{Y},S)$ is ergodic.

Let $(\mathbf{Y}, R)$ be a $G$-space. Consider a metric compact $(Z,d)$ with the Borel probability measure $\eta$, which is invariant under the action of isometries, and let $(g,y)\mapsto S^{g,y}$ (where $y \in \mathbf{Y}$ and $g \in G$) be such a measurable family of isometries that $S^{g,R^h(y)} \circ S^{h,y}=S^{gh,y}$ for every$g,h \in G$ and almost every $y \in Y$, and $S^{e,y}=id$ for almost every $y$ (where $e$ is a group identity an $id$ is an identity map). Then we could define an action of the group $G$ by the formula $Q^g: (y,z) \mapsto (R^g(y), S^{g,y}(x))$ on the set $Y \times Z$ with the product measure. Obviously, this defines dynamical system, which is naturally is extension of $(\mathbf{Y}, R)$. Any extension of $(\mathbf{Y}, R)$ isomorphic to the one obtained in this way is called an {\em isometric extension}.
The following is famous Furstenberg-Zimmer dichotomy(see \cite{Zi76}, \cite{Fu77}, \cite{Gl03}):
\begin{theor}[Furstenberg, Zimmer]
Suppose we have an ergodic extension of an ergodic dynamical system. Then exactly one of the following statements holds:
\begin{enumerate}
\item{extension is weakly mixing}
\item{there is a nontrivial intermediate isometric extension}
\end{enumerate}
\end{theor}
In order to prove the theorem \ref{thm:mixing} we will need the following lemma.
\begin{lem}
Suppose $(\mathbf{X},T)$ is an ergodic isometric extension of ergodic system $(\mathbf{Y},R)$ such that the space         $\mathbf{Y}$ has no atoms. Then $h(\mathbf{X},T) \leq h(\mathbf{Y},R)$. 
\end{lem}
\begin{proof}
Let $h(\mathbf{Y})<\infty$, since otherwise the statement is obvious. 
Let $\mathbf{Y}=(Y,\mathscr{Y}, \nu)$ and $\mathbf{X}=(X,\mathscr{X}, \mu)$. Without loss of generality we can assume that $X=Y \times Z$, there $(Z,d)$ is a metric compact and the action is as
in the definition of an isometric extension. 
Take $\varepsilon>0$. 
Let $\beta$ be a generating partition for $(\mathbf{Y},R)$, such that $H(\beta)<h(\mathbf{Y},R)+\varepsilon/2$. Take $\beta'$ to be corespondent partition of $\mathbf{X}$. By the lemma \ref{lem:generating} there is a a conull set $X' \subset X$, such that for every two points $(y_1,z_1),(y_2,z_2) \in X'$ with $y_1 \neq y_2$, there is such a $g \in G$ that $T^g(y_1,z_1)$ and  $T^g(y_2,z_2)$ lie in the different parts of $\beta'$.
Since $\mathbf{Y}$ has no atoms and it is a standard probability space, we have a sequence $\lbrace C_i \rbrace $ of positive measure subsets in $X$ such that $\nu(C_i) \to 0$. Let us take any point $z_0$ from the support of the measure on the set $Z$. Obviously, any ball $B_r(z_0)$ of positive radius has positive measure. Let   $\lbrace A_i \rbrace$ be the sequence of measurable subsets in $X$ of the form $A_i=C_i \times B_{1/i}(z_0) \subset X \times Z$. Let $\alpha_i = \lbrace A_i, X \setminus A_i\rbrace$ be the sequence of paritions. It is easy to see that $H(\alpha_i) \to 0$. Let us take such a subsequence $n_i \to \infty$
 that $\sum_i{H(\alpha_{n_i})}<\varepsilon/2$, we have that $\beta' \vee \bigvee_i \alpha_{n_i}$ dynamically separates point from the conull subset of $\mathbf{X}$. So it is a generating partition and its entropy is smaller than $h(\mathbf{Y},R)+\varepsilon$. We are done, since $\varepsilon$ could be taken arbitrarily small.
\end{proof}

\begin{proof}[Proof of the theorem \ref{thm:mixing}]
Suppose, on the contrary, the Pinsker factor has a nontrivial intermediate isometric extension. By the previous lemma this extension also has zero entropy, a contradiction.
\end{proof}

\end{document}